\documentclass[11pt]{amsart}

\usepackage[dvipsnames]{xcolor}
\usepackage{amssymb,verbatim}
\usepackage{amsmath,amsfonts,enumitem,bm}
\usepackage[mathscr]{euscript} 
\usepackage{amsthm}
\usepackage{url}
\usepackage{graphicx} 
\usepackage{enumitem} 
\usepackage{hyperref} 
\hypersetup{colorlinks} 
\usepackage{bbm} 
\usepackage{bm} 
\usepackage{mathrsfs} 
\usepackage{cancel} 
\usepackage{ytableau} 
\usepackage{mathtools}
\usepackage{lscape} 
\usepackage{colortbl} 
\usepackage{faktor} 
\usepackage{algorithm}
\usepackage{algorithmic}

\usepackage[colorinlistoftodos]{todonotes}

\RequirePackage{cleveref}
\usepackage{hypcap}
\hypersetup{colorlinks=true, citecolor=darkblue, linkcolor=darkblue}
\definecolor{darkblue}{rgb}{0.0,0,0.7}
\newcommand{\darkblue}{\color{darkblue}}

\definecolor{darkred}{rgb}{0.68,0,0}

\definecolor{darkgreen}{rgb}{0,.38,0}

\newcommand{\defn}[1]{\emph{\darkblue #1}}

\setlist[enumerate]{
	label=\textnormal{({\roman*})},
	ref={\roman*}}

\makeatletter
\def\th@plain{%
	\thm@notefont{}
	\itshape 
}
\def\th@definition{%
	\thm@notefont{}
	\normalfont 
}
\makeatother

\newtheorem{thm}{Theorem}[section]

\newtheorem*{claim*}{Claim}

\newtheorem{prop}[thm]{Proposition}
\newtheorem{conj}[thm]{Conjecture}
\newtheorem{conv}[thm]{Convention}

\theoremstyle{definition}

\newtheorem{rem}[thm]{Remark}

\numberwithin{figure}{section}
\numberwithin{equation}{section}

\def\nn{\mathbb N}

\def\la{\lambda}
\def\ga{\gamma}

\def\de{\delta}

\def\al{\alpha}
\def\be{\beta}

\def\<{\langle}
\def\>{\rangle}

\def\rM{ {\text {\rm M} } }

\def\rD{{\text {\rm D} } }
\def\rC{{\text {\rm C} } }

\def\0{{\mathbf 0}}

\def\.{\hskip.06cm}
\def\ts{\hskip.03cm}

\def\La{\Lambda}

\newcommand{\SYT}{\operatorname{{\rm SYT}}}

\def\.{\hskip.06cm}
\def\ts{\hskip.03cm}

\usepackage{indentfirst}
\DeclareTextSymbolDefault{\ae}{T1}

\begin{document}

\title[On the largest Littlewood--Richardson coefficient]{On the largest Littlewood--Richardson coefficient}

 \author[Igor Pak \. \and \. Daniel Soskin]{Igor Pak$^\star$  \. \and \.  Daniel Soskin$^\star$}

\makeatletter

\thanks{\today}

\thanks{\thinspace ${\hspace{-.45ex}}^\star$Department of Mathematics,
UCLA, Los Angeles, CA 90095, USA;
\texttt{\{pak,dsoskin\}@math.ucla.edu}}

\begin{abstract}
We study partitions which attain the largest Littlewood--Richardson coefficient.
More precisely, we prove that the largest \ts $c^\la_{\mu\nu}$ \ts
is attained at partitions such that \ts $\mu\subseteq \nu$ \ts and \ts $\nu/\mu$ \ts
is a disjoint union of squares.  We conjecture that for \ts $n\ge 16$,
\emph{all} \ts largest \ts $c^\la_{\mu\nu}$ \ts must satisfy this property.
We confirm this conjecture numerically, for  \ts $16 \le n\le 45$.
\end{abstract}

\maketitle

\section{Introduction}\label{s:def}

According to Metropolis \cite[p.~463]{Met80}, it was Paul Stein's idea to use
the MANIAC computer to investigate problems in ``discrete mathematics with
emphasis on combinational theory.''  In Bivins et al.~\cite{BMSW54}, the first publication
that resulted from this investigation, the authors computed character tables of
\ts $S_{15}$ \ts and \ts $S_{16}$, which took 5 and 12 hours, respectively.
In what can only be viewed as gross understatement, the authors commented:
\begin{quote}
``The experience was quite encouraging,
and suggests that it would be profitable
to apply electronic computer techniques
to a large class of quite complicated
problems in algebra and group theory,'' \cite[p.~212]{BMSW54}.
\end{quote}

The story of this paper started with a footnote~$9$ at the end of \cite{BMSW54},
where the authors asked for the largest dimension \ts $\rD(n)$ \ts of the
irreducible $S_n$ module:
$$
\rD(n) \ := \ \max_{\la\vdash n} \. f^\la \quad \text{where} \quad f^\la \, := \, \chi^\la(1) \, = \, |\SYT(\la)|,
$$
where \ts $\SYT(\la)$ \ts is the set of \emph{standard Young tableaux} \ts of shape~$\la$, see e.g.\ \cite{Sag01,Sta99}.

Soon after, Com\'et \cite{Com60} used the \emph{hook-length formula}
to compute \ts $\rD(n)$ \ts for all \ts $n\le 30$.  Baer and Brock \cite{BB68}
extended the computation to \ts $n\le 36$ \ts and asked for the shape of a
partition \ts $\la\vdash n$ \ts on which the maximum is achieved.
This led to a long series of papers computing
the sequence further, see \cite{McK76,KP92,VP10,DS23}.  We refer to
\cite[\href{https://oeis.org/A003040}{A003040}]{OEIS} for the updated
list of values \ts $\{\rD(n), n\le 153\}$.

A theoretical investigation also followed, beginning with McKay's lower bound \cite{McK76},
and the Vershik--Kerov upper bound \cite{VK85}.  Famously, in connection to the study of
\emph{longest increasing subsequences} \ts in random permutations, the asymptotic limit
shape (\emph{VKLS shape}), was determined by Logan--Shepp \cite{LS77} and Vershik--Kerov \cite{VK81}.
Most recently, Aggarwal and Elboim \cite{AE26} proved the \emph{Vershik--Kerov--Pass conjecture}
\ts that
$$
\rD(n) \ = \ \sqrt{n!} \.\cdot\. c^{-\sqrt{n}(1+o(1))} \quad \text{for some} \ \ \. c\. > \. 1.
$$
We refer to \cite{Rom15} for an extensive discussion of the problem and connections to
other areas.

Motivated by this problem, Richard Stanley \cite[Exc.~82(c)]{Sta23}\footnote{This exercise
was added in 2017.} asked the same questions
about the largest \emph{Littlewood--Richardson {\rm (LR)} coefficient}~:
\begin{equation}\label{eq:C}
\rC(n) \ := \ \max_{\la\vdash n} \. \max_{0\le k \le n/2} \. \max_{\mu \vdash k} \. \max_{\nu \vdash n-k} \. c^\la_{\mu\nu}\.,
\end{equation}
see \cite[\href{https://oeis.org/A387851}{A387851}]{OEIS}.
Stanley showed that the maximum is achieved at \ts $k \sim \frac{n}{2}\ts$, when
$$
\rC(n) \ = \ 2^{n/2\ts-\ts O(\sqrt{n})}\,,
$$
and asked about the asymptotic limit shape.  This was resolved by Panova, Yeliussizov
and the first author \cite{PPY19}, who showed that all three partitions have VKLS shape.
The authors also computed the values \ts $\{\rC(n), n \le 23\}$ \ts and
made many additional observations.

In \cite[$\S$4.7]{PPY19}, the authors noted that the maximum in
\eqref{eq:C} is typically achieved on \emph{nested partitions}
\ts $\mu \subseteq \nu \subseteq \la$.  They conjectured that this always holds
\cite[Conj.~4.21]{PPY19}, and observed that the \emph{Lam--Postnikov--Pylyavskyy {\rm (LPP)} inequality}
\ts implies that the nested property holds for all~$n$ on at \emph{least one} \ts largest
LR coefficient.

We computed the values \ts $\{\rC(n), n \le 45\}$, see the Appendix.
Our computations show that the PPY conjecture above fails already for $n=11$,
with \ts $\la=(5,3,2,1)$, $\mu=(4,1)$, $\nu=(3,2,1)$, but \ts $c^\la_{\mu\nu}=\rC(11)=3$.
However, our results suggest that a stronger claim holds for sufficiently large~$n$:

\smallskip

\begin{conj}\label{conj:main}
For all \ts $n\ge 16$, \underline{\em every} maximum in \eqref{eq:C} is achieved on
partitions \ts $\mu \subseteq \nu \subseteq \la$, such that skew shape
\ts $\nu/\mu$ \ts is a disjoint union of squares.  Moreover, for \ts $n\ge 28$, this maximum
is unique up to the transposition\footnote{In our notation, the transposition
can only give the largest LR coefficients, when \ts $\mu$ \ts and \ts $\nu$ \ts have the same size: \ts
$|\mu|=|\nu|=n/2$.}
and conjugation of partitions: \.
$c^\la_{\mu\nu}  \, = \, c^\la_{\nu\mu}  \, = \,  c^{\la'}_{\mu'\nu'}\, = \,  c^{\la'}_{\nu'\mu'}$\..
\end{conj}

\smallskip

This says that \ts $\mu$ \ts can be obtained from \ts $\nu$ \ts by removing some of its corners.
In this paper we present both theoretical and computational evidence in favor of the
conjecture.  Our main result in the following weak version of the conjecture, which, by
the second part of the conjecture, is equivalent to the first part:

\smallskip

\begin{thm}\label{t:main}
For all \ts $n\ge 1$, \underline{\em at least one} maximum in \eqref{eq:C} is achieved on
partitions \ts $\mu \subseteq \nu \subseteq \la$, such that skew shape
\ts $\nu/\mu$ \ts is a disjoint union of squares.
\end{thm}

\smallskip

The new part here is the second condition; without it, this is proved in \cite[Cor.~4.19]{PPY19}.
For example, for
\ts $n=40$, we have \ts $\rC(40) = 1484=c^\la_{\mu\nu}$, where \ts $\la=(9,7,6,5,4,3,2,2,1,1)$,
\ts $\mu = (6,4,3,2,2,1)$, $\nu = (6,5,4,3,2,1,1)$.  Here \ts $\nu/\mu$ \ts is a disjoint
union of $4$ squares, and this is unique largest LR coefficient up to conjugation.

\smallskip
{\small
\begin{rem}
In the example above, we also have \ts $f^\la= \rD(40)\approx 5.9\cdot 10^{22}$ \ts is the
largest dimension as above.
This is too unusual to be a mere coincidence, given that there are exactly two partitions: $\la$
and $\la'$ with the largest dimension~$f^\la$, out of $p(40)=37338$ integer partitions of~$n=40$.  The same
phenomenon holds for a few other values of~$n$. Unfortunately, our dataset is too small
to make predictions whether this happens infinitely often.
The phenomenon is partly explained by \cite[Thm~1.7]{PPY19}, which says that for all maximal
LR coefficients, we must have \ts $f^\la > \sqrt{n!} \. c^{-n}$ \ts for some explicit $c>1$.
\end{rem}}

\smallskip

\section{Schur positive inequalities}\label{s:results}

For standard definitions and notation in algebraic combinatorics and symmetric
functions, see \cite{Sag01,Sta99}.
Recall that the \defn{LR coefficients} \ts are defined
structure constants in the ring $\La$ of symmetric functions:
$$
s_\mu \cdot s_\nu \ = \ \sum_{\la} \. c^\la_{\mu\nu} \. s_\nu\.,
$$
and note that \ts $c^\la_{\mu\nu}=0$ \ts unless \ts $|\la|=|\mu|+|\nu|$.
Famously, LR coefficients \ts $c^\la_{\mu\nu}\in \nn$ \ts have a combinatorial
interpretation which we will not need in this paper.

For two symmetric functions \ts $f,g\in \La$, we write \ts $f\geqslant_s g$ \ts
if \ts $(f-g)$ \ts is \defn{Schur positive}, i.e.\ a linear combination of
Schur functions with nonnegative coefficients.  Recall the \defn{LPP inequality}:
$$
s_\mu \. \cdot \. s_\nu \ \leqslant_s \ s_{\mu \ts \cap \ts \nu} \. \cdot\. s_{\mu \ts\cup\ts \nu}\,,
$$
where \. $\mu \cup \nu$ \. and \. $\mu \cap \nu$ \. are partitions defined by
the union and intersection, respectively, of the corresponding Young diagrams.
This inequality was introduced \cite{LPP07}, and studied extensively in recent years.
See e.g.\ \cite{CCPS26}, for an advanced generalization. In terms of
LR coefficients, the LPP inequality can be restated as
\begin{equation}\label{eq:LPP-LR}
c^\la_{\mu\nu} \, \le \, c^\la_{\mu\cap\nu,\mu\cup\nu} \quad \text{for all} \quad \la,\mu,\nu\ts.
\end{equation}

To prove Theorem~\ref{t:main}, we use the approach in \cite{PPY19},
but this time we also employ the \defn{Okounkov inequality}~$:$
$$
s_{\mu} \. \cdot \. s_{\nu} \ \leqslant_s \ s_{\lfloor(\mu+\nu)/2\rfloor} \. \cdot \. s_{\lceil(\mu+\nu)/2\rceil}\,,
$$
where \. $\lfloor(\al_1,\al_2,\ldots)\rfloor := (\lfloor\al_1\rfloor, \lfloor\al_2\rfloor,\ldots)$
 \. and
\. $\lceil(\al_1,\al_2,\ldots)\rceil := (\lceil\al_1\rceil, \lceil\al_2\rceil,\ldots)$.
A restricted version of this inequality was conjectured by Okounkov \cite{Oko97}.
In full generality, the result was obtained in \cite[Thm~11]{LPP07}.  In terms of
LR coefficients, the Okounkov inequality can be restated as
\begin{equation}\label{eq:Oko-LR}
c^\la_{\mu\nu} \, \le \, c^\la_{\lfloor(\mu+\nu)/2\rfloor,\lceil(\mu+\nu)/2\rceil} \quad \text{for all} \quad \la,\mu,\nu\ts.
\end{equation}

\smallskip
\begin{proof}[Proof of Theorem~\ref{t:main}]
Let \ts $c^\la_{\mu\nu} = \rC(n)$ \ts be a maximal LR coefficient.  By \eqref{eq:LPP-LR},
we also have \. $c^\la_{\mu\cap\nu,\mu\cup\nu} = \rC(n)$.  In other words, the first
condition in the theorem holds for partitions \. $\mu\cap\nu \subseteq \mu\cup\nu \subseteq \la$.

For the second condition, by \eqref{eq:Oko-LR} applied to \. $c^\la_{\al\be}\ts$, where \ts $\al:=\mu\cap\nu$, \ts
$\be :=\mu\cup\nu$, we similarly have: \ts $c^\la_{\lfloor(\al+\be)/2\rfloor,\lceil(\al+\be)/2\rceil}=\rC(n)$.
Note that partitions \ts $\ga:=\lfloor(\al+\be)/2\rfloor$, \ts $\de:= \lceil(\al+\be)/2\rceil$, satisfy
\ts $\ga_i\le \de_i\le \ga_i +1$ \ts for all~$i$.  In other words, the skew shape \ts $\de/\ga$ \ts
is a disjoint union of columns.

By the conjugation symmetry as in the conjecture, we have \ts $c^{\la'}_{\ga'\de'}=\rC(n)$.  By \eqref{eq:Oko-LR}, we also have
\ts $c^\la_{\lfloor(\ga'+\de')/2\rfloor,\lceil(\ga'+\de')/2\rceil}=\rC(n)$.  This gives the desired
partitions \ts $\tau := \lfloor(\ga'+\de')/2\rfloor$, \ts $\pi:=\lceil(\ga'+\de')/2\rceil$, such that
\ts $\tau \subseteq\pi \subseteq\la$ \ts and \ts $\pi/\tau$ \ts is a disjoint union of squares.
\end{proof}

\smallskip

\section{Experimental results}\label{s:exp}

Following \cite{PPY19}, let
\begin{equation}\label{eq:Ck}
\rC(n,k) \ := \ \max_{\la\vdash n} \. \max_{\mu \vdash k} \. \max_{\nu \vdash n-k} \. c^\la_{\mu\nu}\..
\end{equation}
Note that \ts $\rC(n,k)=\rC(n,n-k)$, and
$$
\rC(n) \, = \, \max_{0\le k \le n/2} \. \rC(n,k)\..
$$
It was observed in \cite[$\S$4.8]{PPY19}, that
\ts $\rC(n,k)\le \rC(n+1,k)$, and we also have \defn{stability property} \. $\rC(n,k)=\rD(n)$ \ts for \ts $n\ge \binom{k+1}{2}$.
A table of \ts $\{\rC(n,k)\}$ \ts for \ts $1\le n \le 23$, is given in \cite[App.]{PPY19}

In the Appendix, we present our extensive computations giving new values of \ts $\{\rC(n,k)\}$ \ts
for \ts $24\le n \le 45$.  Our computations confirm the previously computed values, and satisfy
stability, e.g.\ $\rC(45,k) = \rD(k)$ \ts for all \ts $0\le k \le 9$, given that \ts $45=\binom{10}{2}$.
We also computed triples of partitions which attained the maximal value \ts $\rC(n)$, for all \ts $n\le 45$,
and confirmed that Conjecture~\ref{conj:main} hold in these cases, see a discussion below.

The runtimes are summarized in the following table.  For $n\le 40$, the computation was done
on a 4-core laptop.  For $41\le n \le 45$, the computation was done on 32 cores of the
UCLA {\tt Hoffman2} cluster under PyPy~3.11. This explains a big jump down from $n=40$ to~$41$.
Since the case $n=45$ took over 20 hours, computing much beyond that point became infeasible.

\smallskip

{\footnotesize
\begin{table}[htbp]
\centering
\begin{tabular}{|c|c|c|r|}
\hline
\text{$n$}  & $p(n)$ & \text{$\rC(n)$}  &  \text{runtime} \. \\ \hline
25       & 1958         &   45         & 28 \text{sec.}             \\ \hline
30       & 5604       & 176          & 3:18 = 198 \text{sec.}              \\ \hline
35        & 14883       & 423           & 54:18 = 3258 \text{sec.}              \\ \hline
40        & 37338       & 1484          & 12:28:39 = 44919 \text{sec.}           \\ \hline\hline
41 & 44583 & 1768 & 1:33:29 = 5609 \text{sec.}  \\ \hline
45 & 89134 & 4323 & 20:02:52 = 72172 \text{sec.}  \\ \hline
\end{tabular}
\end{table}
}

\smallskip

Here \ts $p(n)=|\{\la \vdash n\}|$ \ts is the number of integer partitions of~$n$, shown
to give the first approximation of the size of the search space: to compute \ts $\rC(n,k)$ \ts
we need to search over \. $p(k)\ts p(n-k) \ts p(n)$ \. triples of partitions.  See
\cite[\href{https://oeis.org/A000041}{A000041}]{OEIS} for further values and references.

\smallskip

\section{Discussion}\label{s:exp-disc}
Let \ts $\rM(n):=|\{\la\vdash n \. : \. f^\la=\rD(n)\}$ \ts denote the number
of irreducible \ts $S_n$ \ts modules of the largest dimension.  It is wide open
whether \ts $\rM(n)$ \ts is bounded, see \cite[Remark~4.2]{PPY19-g}.  Note that
since there are only \ts $p(n)=e^{O(\sqrt{n})}$ \ts partitions of~$n$, while
the maximal dimension \ts $\rD(n)\approx\sqrt{n!}\ts$, one would
certainly expect few if any coincidences (except for conjugate partitions).

In fact, until relatively recently it was open whether \emph{any} \ts dimension, i.e.,
not necessarily maximal, can appear an unbounded number of times; this was proved
by Craven \cite{Cra08}.  In the opposite direction, the bounds in \cite{AE26}
leave open whether \ts $\rM(n)=e^{o(\sqrt{n})}$.

Compare this with the second part of Conjecture~\ref{conj:main}.  Again, asymptotically
there are only \ts $p(n)$ \ts partitions, while the largest LR coefficients is much
larger: \ts $\rC(n) \approx 2^{n/2}$, so one would not expect many coincidences.
Note also that computing LR coefficients efficiently is a major open problem,
see \cite[Conj.~5.14]{Pan24}, compared with \ts $f^\la$ \ts which has a nice
product formula.  Given this state of art, it is rather audacious to believe that
the second part of the conjecture can be resolved in a foreseeable future.

The numerical evidence is helpful at this point.  For example, we have a surprising
coincidence \ts
$c^\la_{\mu\mu}= c^{\la'}_{\mu\mu}=c^\ga_{\mu\mu}= c^{\ga'}_{\mu\mu}=\rC(20)$,
where \ts $\la=(6,5,3,3,2,1)$, \ts $\mu=(4,3,2,1)$ \ts and \ts $\ga=(6,4,4,3,2,1)$.
This shows that the largest LR coefficient is not unique up to transposition and
conjugation in this case.  In our computations, the last time this happens is
for \ts $n=27$, which is why we stated the second part of the conjecture for \ts $n\ge 28$.

\smallskip

For the first part of the conjecture, there is a bit more hope that asymptotic
methods can help, so it can be established without the second part.  We need
the following:

\smallskip

\begin{prop}\label{prop:Ck}
Suppose \ts $\rC(n)$ \ts is unique in the set of values \ts $\{\rC(n,k), 0 \le k \le n/2\}$.
Then the first part of Conjecture~\ref{conj:main} holds.
\end{prop}

\begin{proof}
In notation of the proof of Theorem~\ref{t:main}, suppose the nested property
fails for the maximal LR coefficient \ts $c^\la_{\mu\nu}=\rC(n)$,
i.e.\ suppose $\mu \not \subseteq \nu$.  Then \ts $c^\la_{\al\be}=\rC(n)$ \ts
is also maximal, and since \ts $\mu\subset \al$, \ts $\mu \ne \al$,
we obtain a contradiction with the assumption:  $\rC(n,|\mu|)=\rC(n,|\al|)$.
Use the same argument
for the two averaging operations used in the proof; we omit the details.   \end{proof}

\smallskip

Now, observe that we have only \ts $n/2$ \ts numbers \ts $\{\rC(n,k), 0\le k \le n/2\}$ \ts
distributed between~$1$ \ts and \ts $\rC(n) \approx 2^{n/2}$.  Heuristically, one
would expect a runaway behavior of \ts $\rC(n)$.  Thus, we speculate that the assumption
of the proposition could in principle be proved asymptotically.

Of course, whether the table in the Appendix supports this
approach is a subject for interpretation.  For example, \ts $\rC(45) = \rC(45,22)=4323$ \ts
is much larger than the second largest value \ts $\rC(45,21)=4044$.  On the other hand,
we note that \ts $\rC(42)=\rC(42,18)=2074$, which is only a little larger than \ts
$\rC(42,21)=2064$.

\vskip.7cm
{\small
\subsection*{Acknowledgements}
During the preparation of this manuscript, the authors used
Claude (Anthropic, Claude Opus~4.8) to help with the programming.
We thank the UCLA Advanced Research Computing team for maintaining
the Hoffman2 High-Performance Compute Cluster.  The first author
was partially supported by the NSF grant CCF-2302173.
}

\vskip.8cm

\begin{landscape}

\appendix

\section{Table of the largest Littlewood--Richardson coefficients}\label{s:app}

\vskip.4cm

{\footnotesize
\begin{center}
\begin{tabular}{|c||c|c|c|c|c|c|c|c|c|c|c|c|c|c|c|c|c|c|c|c|c|c|}
\hline
$n$      &  24 &  25 &  26 &  27 &  28 &  29 &  30 &  31 &  32 &  33 &  34 &  35 &  36 &  37 &   38 &   39 &   40 &   41 &   42 &   43 &   44 &   45 \\ \hline
$\rC(n)$ &  41 &  45 &  50 &  59 &  93 & 112 & 176 & 196 & 270 & 298 & 408 & 423 & 640 & 840 & 1096 & 1216 & 1484 & 1768 & 2074 & 2624 & 3450 & 4323 \\ \hline
\end{tabular}
\end{center}
}

\smallskip

\vskip1cm

\begin{center}
Table of \ts $\rC(n,k)$, for \ts $0\le k \le n/2$ \ts and \ts $24 \le n \le 45$.
Green cells are maximal values \ts $\rC(n)$.

\bigskip

\scalebox{1.45}{%
\setlength{\tabcolsep}{2.0pt}%
\scriptsize
\begin{tabular}{|c||c|c|c|c|c|c|c|c|c|c|c|c|c|c|c|c|c|c|c|c|c|c|c|}
\hline
$n\ts\backslash\ts k$ & $0$ & $1$ & $2$ & $3$ & $4$ & $5$ & $6$ & $7$ & $8$ & $9$ & $10$ & $11$ & $12$ & $13$ & $14$ & $15$ & $16$ & $17$ & $18$ & $19$ & $20$ & $21$ & $22$ \\ \hline\hline
$24$ & 1 & 1 & 1 & 2 & 3 & 6 & 16 & 20 & 24 & 26 & 31 & \cellcolor{green!25}41 & 38 &  &  &  &  &  &  &  &  &  &  \\ \hline
$25$ & 1 & 1 & 1 & 2 & 3 & 6 & 16 & 20 & 24 & 29 & 40 & \cellcolor{green!25}45 & 44 &  &  &  &  &  &  &  &  &  &  \\ \hline
$26$ & 1 & 1 & 1 & 2 & 3 & 6 & 16 & 20 & 24 & 29 & 44 & 47 & 48 & \cellcolor{green!25}50 &  &  &  &  &  &  &  &  &  \\ \hline
$27$ & 1 & 1 & 1 & 2 & 3 & 6 & 16 & 20 & 24 & 36 & 48 & \cellcolor{green!25}59 & 58 & 58 &  &  &  &  &  &  &  &  &  \\ \hline
$28$ & 1 & 1 & 1 & 2 & 3 & 6 & 16 & 35 & 30 & 36 & 54 & 64 & 72 & \cellcolor{green!25}93 & 68 &  &  &  &  &  &  &  &  \\ \hline
$29$ & 1 & 1 & 1 & 2 & 3 & 6 & 16 & 35 & 45 & 38 & 60 & 76 & 88 & 111 & \cellcolor{green!25}112 &  &  &  &  &  &  &  &  \\ \hline
$30$ & 1 & 1 & 1 & 2 & 3 & 6 & 16 & 35 & 45 & 55 & 68 & 83 & 110 & 133 & 132 & \cellcolor{green!25}176 &  &  &  &  &  &  &  \\ \hline
$31$ & 1 & 1 & 1 & 2 & 3 & 6 & 16 & 35 & 45 & 60 & 88 & 96 & 112 & 139 & 149 & \cellcolor{green!25}196 &  &  &  &  &  &  &  \\ \hline
$32$ & 1 & 1 & 1 & 2 & 3 & 6 & 16 & 35 & 45 & 60 & 96 & 109 & 134 & 161 & 173 & 218 & \cellcolor{green!25}270 &  &  &  &  &  &  \\ \hline
$33$ & 1 & 1 & 1 & 2 & 3 & 6 & 16 & 35 & 45 & 60 & 96 & 137 & 142 & 168 & 195 & 228 & \cellcolor{green!25}298 &  &  &  &  &  &  \\ \hline
$34$ & 1 & 1 & 1 & 2 & 3 & 6 & 16 & 35 & 45 & 60 & 96 & 146 & 173 & 191 & 203 & 254 & 329 & \cellcolor{green!25}408 &  &  &  &  &  \\ \hline
$35$ & 1 & 1 & 1 & 2 & 3 & 6 & 16 & 35 & 45 & 60 & 96 & 165 & 210 & 213 & 254 & 276 & 377 & \cellcolor{green!25}423 &  &  &  &  &  \\ \hline
$36$ & 1 & 1 & 1 & 2 & 3 & 6 & 16 & 35 & 90 & 75 & 120 & 165 & 210 & 253 & 322 & \cellcolor{green!25}640 & 424 & 462 & 467 &  &  &  &  \\ \hline
$37$ & 1 & 1 & 1 & 2 & 3 & 6 & 16 & 35 & 90 & 111 & 136 & 179 & 249 & 266 & 378 & 744 & \cellcolor{green!25}840 & 566 & 571 &  &  &  &  \\ \hline
$38$ & 1 & 1 & 1 & 2 & 3 & 6 & 16 & 35 & 90 & 111 & 192 & 199 & 268 & 317 & 444 & 864 & 960 & \cellcolor{green!25}1096 & 716 & 693 &  &  &  \\ \hline
$39$ & 1 & 1 & 1 & 2 & 3 & 6 & 16 & 35 & 90 & 111 & 192 & 255 & 277 & 333 & 456 & 882 & 1080 & \cellcolor{green!25}1216 & 1168 & 795 &  &  &  \\ \hline
$40$ & 1 & 1 & 1 & 2 & 3 & 6 & 16 & 35 & 90 & 111 & 192 & 311 & 346 & 387 & 522 & 1002 & 1208 & 1352 & \cellcolor{green!25}1484 & 1424 & 1038 &  &  \\ \hline
$41$ & 1 & 1 & 1 & 2 & 3 & 6 & 16 & 35 & 90 & 111 & 192 & 311 & 414 & 423 & 540 & 1022 & 1328 & 1495 & 1636 & \cellcolor{green!25}1768 & 1536 &  &  \\ \hline
$42$ & 1 & 1 & 1 & 2 & 3 & 6 & 16 & 35 & 90 & 111 & 192 & 347 & 498 & 510 & 618 & 1160 & 1380 & 1666 & \cellcolor{green!25}2074 & 2045 & 1922 & 2064 &  \\ \hline
$43$ & 1 & 1 & 1 & 2 & 3 & 6 & 16 & 35 & 90 & 111 & 192 & 347 & 498 & 600 & 700 & 1186 & 1518 & 1806 & 2332 & \cellcolor{green!25}2624 & 2381 & 2614 &  \\ \hline
$44$ & 1 & 1 & 1 & 2 & 3 & 6 & 16 & 35 & 90 & 111 & 192 & 347 & 498 & 610 & 824 & 1344 & 1664 & 1974 & 2468 & 2882 & 2919 & 3280 & \cellcolor{green!25}3450 \\ \hline
$45$ & 1 & 1 & 1 & 2 & 3 & 6 & 16 & 35 & 90 & 216 & 272 & 347 & 498 & 711 & 872 & 1552 & 2096 & 3250 & 2620 & 3149 & 3490 & 4044 & \cellcolor{green!25}4323 \\ \hline
\end{tabular}
}
\end{center}

\end{landscape}

\end{document}